\documentclass[a4paper,11pt]{amsart}

\usepackage{a4wide}
\usepackage[foot]{amsaddr}

\usepackage{pgfplots}
\usepackage{rotating}

\usepackage{amssymb}
\usepackage{mathrsfs}

\renewcommand{\Re}{\operatorname{Re}}
\renewcommand{\Im}{\operatorname{Im}}

\newcommand{\eq}{:=}

\newcommand{\grad}{\boldsymbol \nabla}
\renewcommand{\div}{\grad \cdot}

\newcommand{\ddiv}{\operatorname{div}}

\newcommand{\BH}{\boldsymbol H}

\newcommand{\BL}{\boldsymbol L}

\newcommand{\ba}{\boldsymbol a}
\newcommand{\bb}{\boldsymbol b}
\newcommand{\bc}{\boldsymbol c}

\newcommand{\be}{\boldsymbol e}

\newcommand{\bn}{\boldsymbol n}
\newcommand{\bo}{\boldsymbol o}

\newcommand{\bv}{\boldsymbol v}

\newcommand{\bx}{\boldsymbol x}

\newcommand{\bz}{\boldsymbol z}

\newcommand{\CF}{\mathcal F}

\newcommand{\CP}{\mathcal P}
\newcommand{\CQ}{\mathcal Q}

\newcommand{\CT}{\mathcal T}

\newcommand{\CV}{\mathcal V}

\newcommand{\LF}{\mathscr F}

\newcommand{\LK}{\mathscr K}

\newcommand{\LP}{\mathscr P}

\newcommand{\LR}{\mathscr R}

\newcommand{\BCP}{\boldsymbol{\CP}}

\newcommand{\BLF}{\pmb{\LF}}

\newcommand{\BLR}{\pmb{\LR}}

\newcommand{\R}{\mathbb R}
\newcommand{\C}{\mathbb C}

\newcommand{\RT}{\boldsymbol{RT}}

\newcommand{\enorm}[1]{|\!|\!|#1|\!|\!|}

\newcommand{\shapereg}{\kappa}

\newcommand{\vel}{\vartheta}
\newcommand{\hmax}{\mathfrak h}
\newcommand{\vmin}{\mathfrak v}

\newcommand{\AAA}{\underline{\mathfrak A}}

\newcommand{\GD}{\Gamma_{\rm D}}
\newcommand{\GN}{\Gamma_{\rm N}}

\newcommand{\sig}{\boldsymbol \sigma}

\newcommand{\MZ}{\underline{\boldsymbol Z}}
\newcommand{\MA}{\underline{\boldsymbol A}}
\newcommand{\MI}{\underline{\boldsymbol I}}
\newcommand{\MW}{\underline{\boldsymbol W}}

\newcommand{\maxest}{\rho_h}
\newcommand{\maxerr}{\varepsilon_h}
\newcommand{\maxcontsol}{\Theta}
\newcommand{\maxsol}{\maxcontsol_h}

\newcommand{\BLFa}{\BLF_h^{\ba}}
\newcommand{\CTa}{\CT_h^{\ba}}
\newcommand{\oma}{\omega^{\ba}}
\newcommand{\pa}{\psi^{\ba}}
\newcommand{\Da}{\mathfrak d^{\ba}}
\newcommand{\Ta}{\mathfrak t^{\ba}}

\usepackage{hyperref}
\hypersetup{
  pdffitwindow=false,
  pdfhighlight=/O,
  pdfnewwindow,
  colorlinks=true,
  citecolor=blue,             
  linkcolor=blue,            
  menucolor=blue,            
  urlcolor=blue,             
  pdfpagemode=UseOutlines,
  bookmarksnumbered=true,
  linktocpage,
  pdfcreator={pdflatex},
  pdfproducer={LaTeX with hyperref}
}

\newtheorem{theorem}{Theorem}
\newtheorem{lemma}[theorem]{Lemma}
\newtheorem{remark}[theorem]{Remark}

\numberwithin{equation}{section}
\numberwithin{theorem}{section}
\numberwithin{figure}{subsection}

\begin{document}

\title[Guaranteed stability bounds]{Guaranteed stability bounds for second-order PDE problems satisfying a G\aa rding inequality}
\author{T. Chaumont-Frelet$^\star$}

\address{\vspace{-.5cm}}
\address{\noindent \tiny \textup{$^\star$Inria, Univ. Lille, CNRS, UMR 8524 -- Laboratoire Paul Painlev\'e}}

\begin{abstract}
We propose an algorithm to numerically determine whether a second-order linear PDE problem
satisfying a G\aa rding inequality is well-posed. This algorithm further provides a
lower bound to the inf-sup constant of the weak formulation, which may in turn be used
for a posteriori error estimation purposes. Our numerical lower bound is based on
two discrete singular value problems involving a Lagrange finite element discretization
coupled with an a posteriori error estimator based on flux reconstruction techniques.
We show that if the finite element discretization is sufficiently rich, our lower bound
underestimates the optimal inf-sup constant only by a factor roughly equal to two at most.
\end{abstract}

\maketitle

\section{Introduction}

Linear boundary value problems with indefinite weak formulations arise in
many important applications including convection-dominated diffusion
and time-harmonic wave propagation problems. In such cases, it is not
always known whether the problem is well-posed. Besides, even in cases
where well-posedness is guaranteed, the magnitude of the stability constant
controlling the norm of the solution in terms of the norm of the right-hand side
is often unknown. In this work, we provide a numerical algorithm that can certify
that the boundary value problem under consideration is well-posed, and
provide a guaranteed upper bound on its stability constant.

We focus on second-order PDE problems of the form: Given $f: \Omega \to \C$,
find $u: \Omega \to \C$ such that
\begin{equation}
\label{eq_bvp_strong}
\left 
\{
\begin{array}{rcll}
-k^2 d u + ik \bc \cdot \grad u-\div (ik\bb u+\MA\grad u) &=& f & \text{ in } \Omega,
\\
u &=& 0 & \text{ on } \GD,
\\
(ik\bb u+\MA\grad u) \cdot \bn &=& 0 & \text{ on } \GN,
\end{array}
\right .
\end{equation}
where $\MA,\bb,\bc$ and $d$ are piecewise constant complex-valued coefficients,
and $\Omega \subset \R^n$ is a bounded domain with $n=2$ or $3$.
The real number $k > 0$ and the complex unit $i$ are conventionally introduced
to make the PDE coefficients physically dimensionless, whereby the dimension of $k$
is the reciprocal of a length. This convention is especially natural for time-harmonic
wave propagation problems where $k$ is the wavenumber, and the coefficients describe
the material properties of the propagation medium. For convection-dominated diffusion
problems only involving real-valued coefficients, the proposed algorithm may be run
employing only real (floating point) numbers.

We demand that the weak formulation of~\eqref{eq_bvp_strong} satisfies a G\aa rding
inequality as stated precisely in~\eqref{eq_gaarding} below. This is for instance always
true if the matrix-coefficient $\MA$ satisfies the positivity property
\begin{equation*}
\Re \MA(\bx) \be \cdot \overline{\be}
\geq
\alpha_\star
>
0
\end{equation*}
for a.e. $\bx$ in $\Omega$ and all unit vectors $\be \in \C^n$. Under this assumption,
we propose an algorithm that provides a guaranteed lower bound $\gamma_h$ to the
inf-sup constant of the sesquilinear form~$\beta(\cdot,\cdot)$ associated
with~\eqref{eq_bvp_strong}. If~\eqref{eq_bvp_strong} is well-posed, we show that
$\gamma_h > 0$ whenever the finite element space employed in the algorithm is
sufficiently rich. This numerically guarantees the well-posedness of~\eqref{eq_bvp_strong},
and leads to upper bounds for the norm of the operator mapping $f$ to $u$ in natural norms.

For simplicity, we assume that the coefficients are piecewise constant onto a
polytopal partition and that the domain and the boundary partition are polytopal.
However, we do not make any regularity assumptions, meaning that the geometry
described by the domain and coefficients can include sharp edges and corners.

The algorithm is based on two discrete singular value problems arising from
a finite element discretization. More specifically, a Lagrange finite element
discretization of~\eqref{eq_bvp_strong} is combined with an a posteriori error
estimator based on a flux reconstruction technique~\cite{braess_schoberl_2008a,%
destuynder_metivet_1999a,ern_vohralik_2015a}. If the problem under consideration is
well-posed, it is guaranteed that the algorithm provides an upper bound for the stability
constant, provided that the finite element space is sufficiently rich.
In fact, we show that as soon as the finite element space provides
reasonable approximate solutions to~\eqref{eq_bvp_strong}, the overestimation on
the stability constant does not exceed roughly a factor two. Furthermore,
the overestimation is independent of the polynomial degree of the finite
element space. This is key for time-harmonic wave propagation problems,
where high-order discretizations are often drastically more performant~\cite{%
ainsworth_2004a,chaumontfrelet_nicaise_2018a,chaumontfrelet_nicaise_2020a,melenk_sauter_2011a}.

Besides their independent interest, guaranteed estimations of the
inf-sup constant are crucial in error certification, as they enter a posteriori
error estimates~\cite{chaumontfrelet_ern_vohralik_2021a,dorfler_sauter_2013a,sauter_zech_2015a}.
As a result, the present result may be combined with existing error estimators
to provide fully-guaranteed error bounds when~\eqref{eq_bvp_strong} is discretized
by finite elements.

The problem under consideration here has already been tackled in the literature
with related ideas, see~\cite{watanabe_kinoshita_nakao_2023a}
and the references therein. However, to the best of the author's knowledge, these works
all require explicit regularity shifts for the principle part of the PDE operator.
In practice, this restricts the setting to convex domains with $\MA = \MI$,
or to domains with smooth boundaries~\cite[Section 6.2.7]{nakao_plum_watanabe_2019a}.
Furthermore, the bounds obtained are not necessarily efficient, especially
for high-order finite element discretizations. In contrast, we employ here
a polynomial-degree-robust a posteriori error estimator which allows us to
work in a general setting where regularity shifts are not available or not
explicit, and to fully exploit the power of high-order finite elements.

Another recent work similar to the present one is~\cite{liu_oishi_2013a}, 
where a discrete eigenvalue problem involving an a posteriori error estimator
based on flux reconstruction techniques is employed. However,~\cite{liu_oishi_2013a}
only focuses on self-adjoint problems, and does not show that the proposed lower
bound is efficient. Besides, poynomial-degree-robustness properties have not
been analyzed in~\cite{liu_oishi_2013a}.

We finally mention that for self-adjoint problems, cheaper algorithms
based on non-conforming or mixed finite element discretizations are
available, see e.g.~\cite{carstensen_gedicke_2014a,gallistl_2023a}.
However, it is not clear that such techniques may be bridged to the
present context.

The remainder of this work is organized as follows.
Section~\ref{section_notation} introduces key notation,
makes the assumptions on~\eqref{eq_bvp_strong}
precise, and collects useful results from the literature.
In Section~\ref{section_lower_bound}, we present our computational
algorithm and establish our guaranteed lower bound.
Section~\ref{section_upper_bound} is dedicated to the
efficiency of the algorithm, whereby we show that our numerical
inf-sup lower bound cannot arbitrarily underestimate the optimal one.
Finally, we present in Section~\ref{section_numerical_examples} two
numerical examples that illustrate the theoretical findings.

\section{Notation, assumptions and tools}
\label{section_notation}

\subsection{Complex numbers}

Classically, we denote by $\R$ and $\C$ the fields of real and complex numbers.
The notation $\R^n$ (resp. $\C^n$) and $\R^{n \times n}$ (resp. $\C^{n \times n}$)
are used for vectors and matrices with real (resp. complex) coefficients. If
$z \in \C$, $z_{\rm r} = \Re z$ and $z_{\rm i} = \Im z$ respectively denote
the real and imaginary parts
of $z$. $z_\dagger$ is its complex conjugate and $|z|$ its modulus. For
a vector $\bz \in \C^d$, $\bz_\dagger$ is its component-wise complex conjugate,
and $|\bz|$ is its $\ell^2(\C^n)$ norm. Finally, if $\MZ \in \C^{n \times n}$
is a matrix, $\MZ_{\rm r}$ and $\MZ_{\rm i}$ are its component-wise real and
imaginary parts. $\MZ_\dagger$ is the adjoint of $\MZ$, i.e., the entries
of $\MZ_\dagger$ are the complex conjugate of the ones of the transpose of $\MZ$.

\subsection{Domain and coefficients}

Throughout this work, $\Omega \subset \R^n$ is a weakly Lipschitz polytopal domain.
The boundary $\partial \Omega$ of $\Omega$ is split into two disjoint relatively
open polytopal subsets $\GD$ and $\GN$ in such way that
$\partial \Omega = \overline{\GD} \cup \overline{\GN}$.

We consider coefficients $\MA: \Omega \to \C^{n \times n}$,
$\bb,\bc: \Omega \to \C^n$ and $d: \Omega \to \C$ that are
piecewise constant on a polytopal partition of $\Omega$.
Specifically, there exists a finite set $\CQ$ of disjoint open
polytopal subsets of $\Omega$ with $\overline{\Omega} = \cup_{Q \in \CQ} \overline{Q}$
such that for all $Q \in \CQ$, there exist constants $\MA_Q \in \C^{n \times n}$,
$\bb_Q,\bc_Q \in \C^n$ and $d_Q \in \C$ such that
\begin{equation}
\label{eq_pw_constant}
\MA(\bx) = \MA_Q,
\quad
\bb(\bx) = \bb_Q,
\quad
\bc(\bx) = \bc_Q,
\quad
d(\bx) = d_Q
\end{equation}
for all $\bx \in Q$.

For simplicity, we assume that the origin belongs to $\Omega$. We then denote by $\ell$
the diameter of $\Omega$ and let $\widehat \Omega \eq (1/\ell) \Omega$.
Throughout the manuscript, $c(\widehat \Omega)$ denotes a constant, that can change from
one occurrence to the other, that only depends on $\widehat \Omega$. Such constant depends
on the ``shape'' of $\Omega$, but not on its size.

\subsection{Function spaces}

For an open set $U \subset \Omega$ with Lipschitz boundary, we denote by $L^2(U)$
the Lebesgue space of (complex-valued) square-integrable functions defined on $U$, and we
let $\BL^2(U) \eq [L^2(U)]^n$. The inner products of both spaces are denoted by
$(\cdot,\cdot)_U$. For measurable weights $w: U \to \R$ and $\MW: U \to \R^{n \times n}$,
we introduce $\|v\|_{w,U} \eq \sqrt{(wv,v)_U}$ and $\|\bv\|_{\MW,U} \eq \sqrt{(\MW\bv,\bv)_U}$
for all $v \in L^2(U)$ and $\bv \in \BL^2(U)$. When $w$ is uniformly away bounded from $0$
and $+\infty$, $\|\cdot\|_{w,U}$ is equivalent to the standard norm on $L^2(U)$. Similarly,
if $\MW$ is symmetric and uniformly bounded from above and below in the sense of quadratic
forms, then $\|\cdot\|_{\MW,U}$ is equivalent to standard norm of $\BL^2(U)$.

The notation $H^1(U)$ is used for the standard Sobolev space of
functions $v \in L^2(U)$ such that $\grad v \in \BL^2(U)$, where
$\grad v$ is the gradient defined in the sense of distributions.
If $\gamma \subset \partial U$ is a relatively open subset of the
boundary, then $H^1_\gamma(U)$ collects functions of $H^1(U) $with
vanishing traces on $\gamma$.

We refer the reader to~\cite{adams_fournier_2003a} for more details on
Lebesgue and Sobolev spaces.

We will finally employ the vector Sobolev space $\BH(\ddiv,U)$ of
vector fields $\bv \in \BL^2(U)$ with weak divergence $\div \bv \in L^2(U)$,
see e.g.~\cite{girault_raviart_1986a}. As above, $\BH_{\gamma}(\ddiv,U)$ is
the subset of $\BH(\ddiv,U)$ consisting of vector fields with vanishing normal
trace on $\gamma$, as per~\cite{fernandes_gilardi_1997a}.

\subsection{Sesquilinear form}

We use the notation $\beta: H^1_{\GD}(\Omega) \times H^1_{\GD}(\Omega) \to \C$
for the sesquilinear form associated with the weak formulation of~\eqref{eq_bvp_strong}.
It is given by
\begin{equation}
\label{eq_definition_beta}
\beta(u,v) \eq (-k^2du+ik\bc \cdot \grad u,v)_\Omega + (ik\bb u + \MA\grad u,\grad v)_\Omega
\end{equation}
for all $u,v \in H^1_{\GD}(\Omega)$. For simplicity, we record here that we
equivalently write that
\begin{equation}
\label{eq_definition_beta_adjoint}
\beta(u,v)
=
(u,-k^2d_\dagger v-ik\bb_\dagger \cdot \grad v)_\Omega
+
(\grad u,\MA_\dagger\grad v-ik\bc_\dagger v)_\Omega.
\end{equation}

For Helmholtz problems without convection, we have $\bb = \bc = \bo$.
In addition, $d$ and $\MA$ are real-valued and positive in the majority
of the domain. These coefficients can have a non-zero imaginary part in
parts of the domain containing absorbing materials, or if a radiation
condition has been approximated by a perfectly matched layer~\cite{berenger_2002a}.

\subsection{G\aa rding inequality}

The key assumption we make throughout this work is that
the sesquilinear form $\beta$ is coercive up to compact
perturbation. Specifically, we assume that there exist
weights
$\mathfrak m, \mathfrak p: \Omega \to \R$ and
$\AAA: \Omega \to \R^{d \times d}$ such that
the G\aa rding inequality
\begin{equation}
\label{eq_gaarding}
\Re \beta(u,u) \geq \enorm{u}_\Omega^2-2k^2\|u\|_{\mathfrak p,\Omega}^2
\end{equation}
holds true with
\begin{equation}
\label{eq_energy_norm}
\enorm{u}_U^2 \eq k^2 \|u\|_{\mathfrak m,U}^2 + \|\grad u\|_{\AAA,U}^2,
\qquad
U \subset \Omega.
\end{equation}
Here, it is assumed that the three weights are piecewise constant on the
partition $\CQ$ as per~\eqref{eq_pw_constant}, that $\mathfrak p \geq 0$,
that $\mathfrak m > 0$ and that $\AAA > 0$ in the sense of quadratic forms.
We also assume for simplicity that $\mathfrak p \not \equiv 0$.

For Helmholtz problems, we can take $\mathfrak p = \mathfrak m = d_{\rm r}$
and $\AAA = \MA_{\rm r}$.


\subsection{Computational mesh}

We consider a mesh $\CT_h$ of the domain $\Omega$ consisting of
(open) simplicial elements $K$. We assume that the mesh is matching,
meaning that the intersection $\overline{K}_+ \cap \overline{K}_-$
of two distinct elements $K_\pm \in \CT_h$ is either empty, or a full
subsimplex (vertex, edge or face) of both elements. We demand that the mesh is conforming,
meaning that the union of the elements cover the domain. We further
require that the coefficients are constant in each element. We also
finally denote by $\CF_h$ the set of mesh faces, and require that every boundary
face either entirely belongs to $\GD$ or to $\GN$.

For an element $K \in \CT_h$, $h_K$ is a diameter of $K$ and $\rho_K$
is the diameter of the largest ball contained in $\overline{K}$. Then,
$\shapereg_K \eq h_K/\rho_K \geq 1$ denote the shape regularity parameter
of $K$, and $\shapereg \eq \max_{K \in \CT_h} \shapereg_K$.

We will often employ the notation $c(\shapereg)$ for a constant,
which may different at each occurrence, only depending only on $\shapereg$.

\subsection{Wavespeed}

For $K \in \CT_h$, we denote by $m_K \eq \mathfrak m|_K$
and $p_K \eq \mathfrak p|_K$ the (constant) restrictions to
$\mathfrak m$ and $\mathfrak p$ to $K$. Similarly,
$\alpha_K^\flat$ and $\alpha_K^\sharp$ denote the smallest
and largest eigenvalues of $\AAA|_K$.
We finally write
\begin{equation}
\label{eq_wavespeed}
\vel_K 
\eq
\sqrt{\frac{p_K}{\alpha_K^\sharp}}
\end{equation}
for the ``wavespeed'' in the element $K$.

\subsection{Polynomial spaces}

If $K \in \CT_h$ is simplex and $r \geq 0$, we denote by
$\CP_r(K)$ the set of (complex-valued) polynomials defined
on $K$ of degree less than or equal to $r$, and we set
$\BCP_r(K) \eq [\CP_r(K)]^d$. We will also need the Raviart--Thomas
polynomial space defined by $\RT_r(K) \eq \BCP_r(K) + \bx \CP_r(K)$,
see~\cite{nedelec_1980a,raviart_thomas_1977a}. If $\CT \subset \CT_h$
is a set of elements, we write $\CP_r(\CT)$, $\BCP_r(\CT)$ and $\RT_r(\CT)$ for functions
whose restriction to each $K \in \CT$ respectively belong to $\CP_r(K)$, $\BCP_r(K)$
and $\RT_r(K)$. Note that these spaces do not embed any compatibility conditions.

\subsection{Finite element spaces}

Throughout, we fix a polynomial degree $p \geq 1$
and consider the Lagrange finite element space
$V_h \eq \CP_p(\CT_h) \cap H^1_{\GD}(\Omega)$.
We will also need an auxiliary space of (discontinuous)
piecewise polynomials. Specificially, we fix
$q \geq 0$ and let $Q_h \eq \CP_q(\CT_h)$.
In practice, we could build $Q_h$ and $V_h$
on different partitions of the mesh, but for
simplicity, we do not. We also note that most
of the proposed analysis is carried out with the
case $q = 0$ in mind, irrespectively of the value
of $p$.

\subsection{Projection}

For $\theta \in L^2(\Omega)$, we denote by $\pi_h \theta \in Q_h$
the orthogonal projection defined by
\begin{equation*}
(\pi_h \theta, r_h)_\Omega = (\theta, r_h)_\Omega
\end{equation*}
for all $r_h \in Q_h$. Classically, this projection is in fact
defined elementwise, and for $K \in \CT_h$,
\begin{equation}
\label{eq_projection_poincare}
\|\theta-\pi_h \theta\|_{K} \leq \frac{h_K}{\pi} \|\grad \theta\|_K,
\end{equation}
whenever $\theta \in H^1(K)$, see e.g.~\cite{bebendorf_2003a}.
Applying~\eqref{eq_projection_poincare} elementwise then gives
\begin{equation}
\label{eq_projection_error}
k\|u-\pi_h u\|_{\mathfrak p,\Omega}
\leq
\frac{k\hmax}{\pi \vmin} \|\grad u\|_{\AAA,\Omega}
\end{equation}
for all $u \in H^1(\Omega)$,
where $\hmax \eq h_{K_\star}$ and $\vmin \eq \vel_{K_\star}$
for (one of) the element(s) $K_\star \in \CT_h$ such that
\begin{equation*}
\frac{h_{K_\star}}{\vel_{K_\star}} = \max_{K \in \CT_h} \frac{h_K}{\vel_K}.
\end{equation*}
Finally, because $\mathfrak m$ is piecewise constant, $\pi_h$ is also an orthogonal
projection in the $\mathfrak m$-weighted $L^2(\Omega)$ inner-product, and we have
\begin{equation}
\label{eq_projection_norm}
\|\pi_h \theta\|_{\mathfrak m,\Omega}
\leq
\|\theta\|_{\mathfrak m,\Omega}.
\end{equation}

\section{Guaranteed inf-sup lower bound}
\label{section_lower_bound}

We are now ready to describe our algorithm. It relies on the fact
that the finite element discretization with the space $V_h$ to~\eqref{eq_bvp_strong}
is well-posed (for sufficiently rich spaces), and is based on two discrete singular
value problems involving the space $Q_h$.

\subsection{Discrete Solution operator}

We assume that for all $\theta_h \in Q_h$, there exists a unique $\LP_h \theta_h \in V_h$
such that
\begin{equation}
\label{eq_definition_LPh}
\beta(w_h,\LP_h(\theta_h)) = k^2 (\mathfrak p w_h,\theta_h)_\Omega
\end{equation}
for all $w_h \in V_h$. We then introduce
\begin{equation*}
\maxsol
\eq
\max_{\substack{\theta_h \in Q_h \\ k\|\theta_h\|_{\mathfrak m} = 1}}
\enorm{\LP_h(\theta_h)}_\Omega.
\end{equation*}
The constant $\maxsol$ can be computed as the solution to matrix singular value
problem. In practice, $\maxsol$ is not exactly computable, but guaranteed
upper bound of arbitrary accuracy may be numerically evaluated,
see~\cite[Chapter 12]{nakao_plum_watanabe_2019a}.
As we will see $\maxsol$ is the key ingredient of our inf-sup lower bound.
Specifically, $1/(1+2\maxsol)$ is a satisfactory bound if the mesh is sufficiently
fine.

\subsection{Error estimator}

We rely on a posteriori error estimation to detect whether the mesh is
sufficiently fine to trust the bound based on $\maxsol$. We will call a
flux reconstruction any linear map $\BLF_h: Q_h \to \BH_{\GN}(\ddiv,\Omega)$
such that
\begin{equation}
\label{eq_equilibration}
\div \BLF_h(\theta_h) = k^2 \mathfrak p \theta_h+k^2d_\dagger \LP_h(\theta_h)+ik\bb_\dagger \cdot \grad \LP_h(\theta_h)
\end{equation}
for all $\theta_h \in Q_h$.
For shortness, we also introduce
\begin{equation*}
\BLR_h(\theta_h) \eq \MA_\dagger \grad \LP_h(\theta_h)-ik\bc_\dagger \LP_h(\theta_h)+\BLF_h(\theta_h),
\end{equation*}
and
\begin{equation}
\label{eq_definition_maxest}
\maxest \eq \max_{\substack{\theta_h \in Q_h \\ k\|\theta_h\|_{\mathfrak p} = 1}}
\|\BLR_h(\theta_h)\|_{\AAA^{-1},\Omega}.
\end{equation}
As for $\maxsol$, the constant $\maxest$ can be computed (or at least, rigorously estimated
from above) via the numerical solution of a discrete singular value problem.

\subsection{Lower bound}

Our numerical algorithm simply amounts to computing $\maxsol$ and $\maxest$.
As we now establish, these two constants may be combined in a simple
algebraic expression to provide a lower bound to the inf-sup constant
of $\beta$. The proposed algorithm works for any choice of flux reconstruction
$\BLF$ satisfying~\eqref{eq_equilibration}. A possible construction will be given
in Section~\ref{section_flux_reconstruction} below.

We start with a Prager--Synge type estimate.
This result is standard, see e.g.~\cite{ern_vohralik_2015a,prager_synge_1947a,vohralik_2010a},
but have not been established for the particular setting considered here, in
particular because the matrix coefficient $\MA$ is complex-valued. We therefore
include a proof for completeness.

\begin{lemma}[Control of the residual]
For all $\theta_h \in Q_h$, the estimate
\begin{equation*}
\max_{\substack{w \in H^1_{\GD}(\Omega) \\ \|\grad w\|_{\AAA,\Omega} = 1}}
|k^2 (\mathfrak p w,\theta_h)_\Omega-\beta(w,\LP_h(\theta_h))|
\leq
\|\BLR(\theta_h)\|_{\AAA^{-1},\Omega}
\end{equation*}
holds true. In particular, we have
\begin{equation}
\label{eq_estimate_rho}
\max_{\substack{\theta_h \in Q_h \\ k\|\theta_h\|_{\mathfrak p,\Omega} = 1}}
\max_{\substack{w \in H^1_{\GD}(\Omega) \\ \|\grad w\|_{\AAA,\Omega} = 1}}
|k^2 (\mathfrak p w,\theta_h)_\Omega-\beta(w,\LP_h(\theta_h))|
\leq
\maxest.
\end{equation}
\end{lemma}

\begin{proof}
Fix $\theta_h \in Q_h$ and let $u_h \eq \LP_h(\theta_h)$, $\sig_h \eq \BLF_h(\theta_h)$.
In view of~\eqref{eq_definition_beta_adjoint} and~\eqref{eq_equilibration}, we have
\begin{align*}
k^2 (\mathfrak p w,\theta_h)_\Omega-\beta(w,u_h)
&=
(w,k^2 \mathfrak p \theta_h + k^2 d_\dagger u_h+ik\bb_\dagger \cdot \grad u_h)_\Omega
-
(\grad w,\MA_\dagger \grad u_h-ik\bc_\dagger u_h)_\Omega
\\
&=
(w,\div \sig_h)_\Omega
-
(\grad w,\MA_\dagger \grad u_h-ik\bc_\dagger u_h)_\Omega
\\
&=
-
(\grad w,\MA_\dagger \grad u_h-ik\bc_\dagger u_h+\sig_h)_\Omega,
\end{align*}
where use integration by part in the last identity.
We conclude with a Cauchy--Schwarz inequality that
\begin{align*}
|k^2 (\mathfrak pw,\theta_h)_\Omega-\beta(w,u_h)|
&=
|(\AAA \grad w, \AAA^{-1}(\MA_\dagger \grad u_h-ik\bc_\dagger u_h+\sig_h))_\Omega|
\\
&\leq
\|\grad w\|_{\AAA,\Omega}\|
\|\AAA^{-1}(\MA_\dagger \grad u_h-ik\bc_\dagger u_h+\sig_h)\|_{\AAA,\Omega}
\\
&=
\|\grad w\|_{\AAA,\Omega}\|
\|\MA_\dagger \grad u_h-ik\bc_\dagger u_h+\sig_h\|_{\AAA^{-1},\Omega},
\end{align*}
from which the conclusion follows.
\end{proof}

We now establish our guaranteed lower bound for the inf-sup constant of $\beta$.

\begin{theorem}[Guaranteed bounds]
The lower bound
\begin{equation}
\label{eq_T_coercivity}
\Re \beta(u,u+2\LP_h(\pi_h u))
\geq
\left \{
1
-
2\left (\frac{k\hmax}{\vmin}\right )^2
-
2\maxest
\right \}
\enorm{u}_\Omega^2
\end{equation}
holds true for all $u \in H^1_{\GD}(\Omega)$. In particular
\begin{equation}
\label{eq_lower_bound}
\min_{\substack{u \in H^1_{\GD}(\Omega) \\ \enorm{u}_\Omega = 1}}
\max_{\substack{v \in H^1_{\GD}(\Omega) \\ \enorm{v}_\Omega = 1}}
\Re \beta(u,v)
\geq
\gamma_h,
\end{equation}
with
\begin{equation}
\label{eq_definition_gamma_h}
\gamma_h
\eq
\left \{
1
-
2\left (\frac{k\hmax}{\vmin}\right )^2
-
2\maxest
\right \}
\frac{1}{1+2\maxsol}.
\end{equation}
\end{theorem}

\begin{proof}
Considering an arbitrary $u \in H^1_{\GD}(\Omega)$, we start with the G\aa rding
inequality stated in~\eqref{eq_gaarding}, namely
\begin{equation*}
\Re \beta(u,u) \geq \enorm{u}_\Omega^2 - 2k^2\|u\|_{\mathfrak p,\Omega}^2.
\end{equation*}
We then use~\eqref{eq_projection_error}, showing that
\begin{equation*}
k^2 \|u\|_{\mathfrak p,\Omega}^2
=
k^2 \|\pi_h u\|_{\mathfrak p,\Omega}^2
+
k^2\|u-\pi_h u\|_{\mathfrak p,\Omega}^2
\leq
\left (
\frac{k\hmax}{\pi \vmin}
\right )^2
\enorm{u}_\Omega^2
+
k^2 \|\pi_h u\|_{\mathfrak p,\Omega}^2
\end{equation*}
from which we infer
\begin{equation}
\label{tmp_inf_sup_gaarding}
\Re \beta(u,u)
\geq
\left \{
1-
2\left (
\frac{k\hmax}{\pi \vmin}
\right )^2
\right \}\enorm{u}_\Omega^2 - 2k^2\|\pi_h u\|_{\mathfrak p,\Omega}^2.
\end{equation}
We now invoke~\eqref{eq_estimate_rho}, which allows us to write that
\begin{multline*}
|k^2 (\mathfrak p w,\pi_h u)_\Omega-\beta(w,\LP_h(\pi_h u))|
\leq
\maxest \|\grad w\|_{\AAA,\Omega} k\|\pi_h u\|_{\mathfrak m,\Omega}
\\
\leq
\maxest \|\grad w\|_{\AAA,\Omega} k\|u\|_{\mathfrak m,\Omega}
\leq
\maxest \enorm{w}_\Omega \enorm{u}_\Omega,
\end{multline*}
for all $w \in H^1_{\GD}(\Omega)$ and from which we deduce that
\begin{equation}
\label{tmp_inf_sup_correction}
\Re \beta(u,\LP_h(\pi_h u))
\geq
k^2 \|\pi_h u\|_{\mathfrak p,\Omega}^2
-
\maxest \enorm{u}_\Omega^2.
\end{equation}
At this point~\eqref{eq_T_coercivity} follows by adding twice~\eqref{tmp_inf_sup_correction}
to~\eqref{tmp_inf_sup_gaarding}, since these estimates holds for all $u \in H^1_{\GD}(\Omega)$.

To establish~\eqref{eq_lower_bound} from~\eqref{eq_T_coercivity}, we first fix
$u \in H^1_{\GD}(\Omega)$ and observe that picking $v^\star \eq u + 2\LP_h(\pi_h u)$,
we have
\begin{equation*}
\max_{\substack{v \in H^1_{\GD}(\Omega) \\ \enorm{v}_\Omega = 1}}
\Re \beta(u,v)
\geq
\frac{1}{\enorm{v^\star}_\Omega} \Re \beta(u,v^\star)
\geq
\frac{1}{\enorm{v^\star}_\Omega}
\left \{
1
-
2\left (\frac{k\hmax}{\vmin}\right )^2
-
2\maxest
\right \}
\enorm{u}_\Omega^2
\end{equation*}
Then the desired estimate follows from~\eqref{eq_T_coercivity}
together with the fact that
\begin{equation*}
\enorm{v^\star}_\Omega
\leq
\enorm{u}_\Omega + 2\maxsol k\|\pi_h u\|_{\mathfrak m,\Omega}
\leq
\enorm{u}_\Omega + 2\maxsol k\|u\|_{\mathfrak m,\Omega}
\leq
(1+2\maxsol)\enorm{u}_\Omega,
\end{equation*}
where we employed~\eqref{eq_projection_norm}.
\end{proof}

\section{Efficiency}
\label{section_upper_bound}

In this section, we show that the lower bound proposed
in~\eqref{eq_lower_bound} is efficient. By that, we mean
that if $\beta$ is indeed inf-sup stable, the numerical
lower bound is $\gamma_h$ positive and does not arbitrarily underestimate
the optimal inf-sup constant $\gamma$, provided the finite element
space $V_h$ is sufficiently rich and that the flux reconstruction
$\BLF$ is suitably designed.

From here on, we therefore assume that $\beta$ is inf-sup stable,
namely that
\begin{equation}
\label{eq_inf_sup_beta}
\gamma
\eq
\min_{\substack{u \in H^1_{\GD}(\Omega) \\ \enorm{u}_\Omega = 1}}
\max_{\substack{v \in H^1_{\GD}(\Omega) \\ \enorm{v}_\Omega = 1}}
\Re \beta(u,v)
>
0
\end{equation}
We will also denote by
\begin{equation}
\label{eq_continuity_beta}
M
\eq
\max_{\substack{u \in H^1_{\GD}(\Omega) \\ \enorm{u}_\Omega = 1}}
\max_{\substack{v \in H^1_{\GD}(\Omega) \\ \enorm{v}_\Omega = 1}}
|\beta(u,v)|
\end{equation}
the continuity constant of $\beta$ in the chosen energy norm.
We can then introduce the continuous solution operator
\begin{equation}
\label{eq_definition_LP}
b(w,\LP(\theta_h)) = k^2(\mathfrak p w,\theta_h)
\end{equation}
for all $\theta_h$ and $w \in H^1_{\GD}(\Omega)$.

For Helmholtz problems, $M$ is bounded from above by a generic
$k$-independent constant. In the absence of dissipation, we
usually have $M = 1$. Otherwise, it depends on the strength of
the absorption, or on the parameters of the perfectly matched layers
when they are employed.

\subsection{Vertex patches}

In this section, we denote by $\CV_h$ the set of vertices of the mesh $\CT_h$.
For each $\ba \in \CV_h$, we denote by $\pa \in \CP_1(\CT_h) \cap H^1(\Omega)$
its hat function, i.e., this only continuous piecewise affine function such that
$\pa(\bb) = \delta_{\ba,\bb}$ for all $\bb \in \CV_h$, where $\delta$ stands
for the Kronecker symbol. We denote by $\CTa \subset \CT_h$ the set of elements
having $\ba$ as a vertex. Then, the open domain covered by the elements of $\CTa$
is denoted by $\oma$, and corresponds to the support of $\pa$.

\subsection{Local wavespeed and contrast}

For all $\ba \in \CV_h$, we let
\begin{equation*}
\label{eq_wavespeed_oma}
\vel_{\oma}
\eq
\sqrt{\frac{\min_{K \in \CTa} p_K}{\max_{K \in \CTa} \alpha_K^\sharp}},
\qquad
\LK_{\oma}
\eq
\sqrt{\frac{\max_{K \in \CTa} \alpha_K^\sharp}{\min_{K \in \CTa} \alpha_K^\flat}}.
\end{equation*}
We also denote by $h_{\oma}$ the diameter of $\oma$.

\subsection{Local function spaces}

The following spaces associated to vertex patches will be useful.
For $\ba \in \CV_h$, we let $\gamma_{\ba}^{\rm c} \subset \partial \oma$
be the set covered by the faces $F \in \CF_h$ that share the vertex $\ba$
such that $F \subset \GN$. We note that for interior vertices
$\gamma_{\ba}^{\rm c} = \emptyset$. We also let
$\gamma_{\ba} \eq \partial \oma \setminus \gamma_{\ba}^{\rm c}$.
We then let
$\BH_0(\ddiv,\oma) \eq \BH_{\gamma_{\ba}}(\ddiv,\oma)$.
We further let $L^2_0(\oma) \eq \div \BH_0(\ddiv,\oma)$. This space coincides with
$L^2(\oma)$ if $\gamma_{\ba}^{\rm c} \neq \emptyset$, and consists of zero mean value
functions otherwise.

\subsection{Localized flux reconstruction}
\label{section_flux_reconstruction}

We are now in place to propose a concrete strategy to compute
a flux reconstruction $\BLF_h: Q_h \to \RT_{p+2}(\CT_h) \cap \BH_{\GN}(\ddiv,\Omega)$
satisfying~\eqref{eq_equilibration}. It is defined through the solve
of vertex patch mixed finite element problems.

Given $\theta_h \in Q_h$, for all vertices $\ba \in \CT_h$, we introduce the
divergence constraint
\begin{multline*}
\Da(\theta_h)
\eq
\pa (k^2 \mathfrak p \theta_h + k^2 d_\dagger \LP_h(\theta_h) + ik\bb_\dagger \cdot \grad \LP_h(\theta_h))
\\
-
\grad \pa \cdot (-ik\bc_\dagger\LP_h(\theta_h)+\MA_\dagger\grad \LP_h(\theta)) \in \CP_{p+2}(\CTa)
\end{multline*}
and the target
\begin{equation*}
\Ta(\theta_h)
\eq
\pa (\MA_\dagger \grad \LP_h(\theta_h)-ik\bc_\dagger \LP_h(\theta_h))
\in \BCP_{p+1}(\CTa).
\end{equation*}
These data enter the construction of $\BLF_h$ as follows.
\begin{subequations}
\label{eq_definition_BLF}
For all $\ba \in \CV_h$, we will see below that
\begin{equation}
\label{eq_definition_BLFa}
\BLFa(\theta_h)
\eq
\arg
\min_{\substack{
\sig_h \in \RT_{p+2}(\CTa) \cap \BH_0(\ddiv,\oma)
\\
\div \sig_h(\theta_h)
=
\Da
}}
\|\sig_h+\Ta(\theta_h)\|_{\AAA^{-1},\oma}
\end{equation}
is a sound defintion. Whenever useful, we will also implicitely extend $\BLFa(\theta_h)$
by $\bo$ to $\Omega$, which produces an element of $\BH_{\GN}(\ddiv,\Omega)$. We then let
\begin{equation}
\BLF_h(\theta_h)
\eq
\sum_{\ba \in \CV_h} \BLFa(\theta_h) \in \BH_{\GN}(\ddiv,\Omega).
\end{equation}
\end{subequations}

Before deriving key properties of $\BLF_h$, we immediately make a remark useful
at different places.

\begin{lemma}[Data identity]
For all $\theta_h \in Q_h$, $\ba \in \CV_h$ and $w \in H^1(\oma)$, we have
\begin{equation}
\label{eq_data_identity}
b(\pa w,\LP(\theta_h)-\LP_h(\theta_h)) = (\grad w,\Ta(\theta_h))_{\oma}-(w,\Da(\theta_h))_{\oma}.
\end{equation}
\end{lemma}

\begin{proof}
For shortness, we let $u_h \eq \LP_h(\theta_h)$. Then, we have
\begin{equation*}
(\grad w,\Ta(\theta_h))_{\oma}
=
(\grad w,\pa (\MA_\dagger \grad u_h-ik\bc_\dagger u_h))_{\oma}
=
(\pa\grad w,\MA_\dagger \grad u_h-ik\bc_\dagger u_h)
\end{equation*}
and
\begin{align*}
(w,\Da(\theta_h))_{\oma}
&=
(w,\pa (k^2 \mathfrak p \theta_h + k^2 d_\dagger u_h+ik\bb_\dagger \cdot \grad u_h))_{\oma}
-
(w,\grad \pa \cdot (-ik\bc_\dagger u_h+\MA_\dagger \grad u_h))_{\oma}
\\
&=
(\pa w,k^2 \mathfrak p \theta_h) + (\pa w,k^2 d_\dagger u_h+ik\bb_\dagger \cdot \grad u_h)
-
(w\grad \pa,-ik\bc_\dagger u_h+\MA_\dagger \grad u_h).
\end{align*}
Using the product rule $\grad (\pa w) = \pa \grad w + w\grad \pa$, we have
\begin{align*}
(\grad w,\Ta)_{\oma}
-
(w,\Da)_{\oma}
&=
(\pa w,k^2 \mathfrak p\theta_h)
\\
&-
\left \{
(\pa w,-k^2 d_\dagger u_h-ik\bb_\dagger \cdot \grad u_h)
+
(\grad (\pa w),\MA_\dagger\grad u_h-ik\bc_\dagger u_h)
\right \}
\\
&=
k^2 (\mathfrak p\pa w,\theta_h) - b(\pa w,\LP_h(\theta_h))
\\
&=
b(\pa w,\LP(\theta_h)-\LP_h(\theta_h)),
\end{align*}
where we used the expression for $\beta$ in~\eqref{eq_definition_beta_adjoint}.
\end{proof}

\subsection{Efficiency of the flux reconstruction}

We can now show that the flux reconstruction in~\eqref{eq_definition_BLF} 
lead to a small residual $\BLR(\theta_h)$ whenever the finite element error
$(\LP-\LP_h)(\theta_h)$ is small.

\begin{lemma}[Discrete stable minimization]
For all $\theta_h \in Q_h$, the definition of $\BLFa(\theta_h)$ in~\eqref{eq_definition_BLFa}
is well-posed. $\BLFa(\theta_h)$ depends linearly on $\theta_h$, and we have
\begin{equation}
\label{eq_stable_minimization}
\|\BLFa(\theta_h)+\Ta(\theta_h)\|_{\AAA^{-1},\oma}
\leq c(\kappa)
\min_{\substack{
\sig \in \BH_0(\ddiv,\oma)
\\
\div \sig
=
\Da(\theta_h)
}}
\|\sig+\Ta(\theta_h)\|_{\AAA^{-1},\oma}.
\end{equation}
\end{lemma}

\begin{proof}
Following~\cite{braess_pillwein_schoberl_2009a,chaumontfrelet_vohralik_2024a,ern_vohralik_2021a},
the well-posedness of~\eqref{eq_definition_BLFa} and the estimate in~\eqref{eq_stable_minimization}
follow if we can show that the compatibility condition
\begin{equation*}
(1,\Da(\theta_h))_{\oma} = 0
\end{equation*}
holds true for all vertices $\ba \in \CV_h \setminus \overline{\GD}$.
To do so, we simply invoke~\eqref{eq_data_identity}, giving
\begin{equation*}
(1,\Da(\theta_h))_{\oma} = -b(\pa,\LP(\theta_h)-\LP_h(\theta_h)).
\end{equation*}
Due to the respective definitions of of $\LP(\theta_h)$ and $\LP_h(\theta_h)$
in~\eqref{eq_definition_LP} and~\eqref{eq_definition_LPh}, the right-hand side vanishes
since $\pa \in V_h$. This concludes the proof.
\end{proof}

\begin{lemma}[Local efficiency]
For all $\theta_h \in Q_h$ and $\ba \in \CV_h$, we have
\begin{equation}
\label{eq_local_efficiency}
\min_{\substack{
\sig \in \BH_0(\ddiv,\oma)
\\
\div \sig
=
\Da
}}
\|\sig+\Ta\|_{\AAA^{-1},\oma}
\leq c(\kappa)
M
\left (
\frac{kh_{\oma}}{p\vel_{\oma}}
+
\LK_{\oma}
\right )
\enorm{\LP(\theta_h)-\LP_h(\theta_h)}_{\oma}.
\end{equation}
\end{lemma}

\begin{proof}
The Euler--Lagrange equations defining the miminizer in~\eqref{eq_local_efficiency}
consists in finding $\sig \in \BH_0(\ddiv,\oma)$ and $\xi \in L^2_0(\oma)$ such that
\begin{equation*}
\left \{
\begin{array}{rcll}
(\AAA^{-1} \bv,\sig)_{\oma} - (\div \bv,\xi)_{\oma}
&=&
-(\AAA^{-1} \bv,\Ta)_{\oma}
&
\forall \bv \in \BH_0(\ddiv,\oma),
\\
(w,\div \sig)_{\oma}
&=&
(q,\Da)_{\oma}
&
\forall w \in L^2_0(\oma).
\end{array}
\right .
\end{equation*}
From the first equation, we infer that $\xi \in H^1_{\gamma_{\ba}^{\rm c}}(\oma)$ with
$\grad \xi = \AAA^{-1}(\sig+\Ta)$,
and therefore
\begin{equation*}
\|\sig+\Ta\|_{\AAA^{-1},\oma}
=
\|\grad \xi\|_{\AAA,\oma}.
\end{equation*}
By using a test function $w \in H^1_{\gamma_{\ba}^{\rm c}}(\oma) \cap L^2_0(\oma)$
in the second equation, we have
\begin{equation}
\label{tmp_equation_xi}
(\AAA\grad \xi,\grad w)_{\oma}
=
(\Ta,\grad w)_{\oma}
+
(\sig,\grad w)_{\oma}
=
(\Ta,\grad w)_{\oma}
-
(\Da,w)_{\oma}.
\end{equation}
Recalling~\eqref{tmp_equation_xi} and the Galerkin orthogonality property
satisfied by $\LP_h(\theta_h)$, it follows that
\begin{align*}
\|\grad \xi\|_{\AAA,\oma}^2
&=
b(\pa \xi,\LP(\theta_h)-\LP_h(\theta_h)),
\\
&=
b(\pa \xi-\pa J_h\xi,\LP(\theta_h)-\LP_h(\theta_h)),
\\
&\leq
M \enorm{\LP(\theta_h)-\LP_h(\theta_h)}_{\oma}\enorm{\pa \xi-\pa J_h\xi}_{\oma}.
\end{align*}
where
$J_h: H^1_{\gamma_{\ba}^{\rm c}}(\oma) \to H^1_{\gamma_{\ba}^{\rm c}}(\oma) \cap \CP_{p-1}(\CTa)$
is the quasi-interpolation operator from~\cite{karkulik_melenk_2015a}. We can then
write on the one hand that
\begin{equation*}
k\|\pa \xi-\pa J_h \xi\|_{\mathfrak p,\oma}
\leq
k\|\xi-J_h \xi\|_{\mathfrak p,\oma}
\leq c(\kappa)
\frac{kh_{\oma}}{p \vel_{\oma}}\|\grad \xi\|_{\AAA,\oma}
\end{equation*}
and on the other hand that
\begin{equation*}
\|\grad(\pa \xi-\pa J_h\xi)\|_{\AAA,\oma}
\leq c(\kappa)
\max_{K \in \CTa} \sqrt{\alpha_{K}^\sharp}
\left (
h_{\oma}^{-1} \|\xi-J_h \xi\|_{\oma}
+
\|\grad(\xi-J_h\xi)\|_{\oma}
\right )
\leq c(\kappa)
\LK_{\oma} \|\grad \xi\|_{\AAA,\oma}.
\end{equation*}
Combining these bounds gives~\eqref{eq_local_efficiency}.
\end{proof}

\begin{theorem}[Efficiency of the residual control]
For all $\theta_h \in Q_h$, we have
\begin{equation}
\label{eq_global_efficiency}
\|\BLR(\theta_h)\|_{\AAA^{-1},\Omega}
\leq c(\kappa)
M
\max_{\ba \in\CV_h}
\left (
\LK_{\oma} + \frac{kh_{\oma}}{p \vel_{\oma}}
\right )
\enorm{(\LP-\LP_h)(\theta_h)}_\Omega.
\end{equation}
In addition, the estimate
\begin{equation}
\label{eq_efficiency_maxest}
\maxest
\leq c(\kappa)
M
\max_{\ba \in\CV_h}
\left (
\LK_{\oma} + \frac{kh_{\oma}}{p \vel_{\oma}}
\right )
\maxerr
\end{equation}
holds true, where
\begin{equation}
\label{eq_definition_maxerr}
\maxerr
\eq
\max_{\substack{\theta_h \in Q_h \\ k\|\theta_h\|_{\mathfrak m} = 1}}
\enorm{(\LP-\LP_h)(\theta_h)}_\Omega.
\end{equation}
\end{theorem}

\begin{proof}
Let $\theta_h \in Q_h$. By the definition of $\LP$ in~\eqref{eq_definition_LP}
and invoking the continuity of $\beta$ in~\eqref{eq_continuity_beta}, we have
\begin{equation*}
|(\mathfrak p w,\theta_h)-\beta(w,\LP_h(\theta_h))|
=
|\beta(w,(\LP-\LP_h)(\theta_h))|
\leq
M \enorm{w}\enorm{(\LP-\LP_h)(\theta_h)}
\leq
M \maxerr \enorm{w},
\end{equation*}
and the conclusion follows from the definition of $\maxest$ in~\eqref{eq_definition_maxest}.
\end{proof}

\subsection{Upper bound}

We introduce
\begin{equation}
\label{eq_definition_maxcontsol}
\maxcontsol
\eq
\max_{\substack{\theta \in L^2(\Omega) \\ k\|\theta\|_{\mathfrak m} = 1}}
\enorm{\LP(\theta)}_\Omega,
\end{equation}
the continuous counterpart to $\maxsol$. From the definition of~$\maxerr$
in~\eqref{eq_definition_maxerr}, it is immediate that
\begin{equation}
\label{eq_bound_maxsol}
\maxsol \leq \maxcontsol + \maxerr.
\end{equation}
For Helmholtz problems, it is known that $\maxcontsol$ grows at least linearly with
the wavenumber, see e.g.~\cite{chaumontfrelet_gallistl_nicaise_tomezyk_2022a,%
galkowski_spence_wunsch_2020a}, so that this constant is expected to be large in
the cases of interest.

\begin{lemma}[Inf-sup upper bound]
Assume that $\beta$ is symmetric in the sense that
\begin{equation}
\label{eq_beta_symmetric}
\beta(u,v) = \beta(\overline{v},\overline{u})
\end{equation}
for all $u,v \in H^1_{\GD}(\Omega)$. Then, we have
\begin{equation}
\label{eq_upper_bound_infsup}
\min_{\substack{u \in H^1_{\GD}(\Omega) \\ \enorm{u}_\Omega = 1}}
\max_{\substack{v \in H^1_{\GD}(\Omega) \\ \enorm{v}_\Omega = 1}}
\Re \beta(u,v)
\leq
\frac{\mathfrak K}{\maxcontsol}
\end{equation}
where
\begin{equation*}
\mathfrak K \eq \max_{K \in \CT_h} \sqrt{\frac{p_K}{m_K}}.
\end{equation*}
\end{lemma}

\begin{proof}
Let $\theta \in L^2(\Omega)$ denote a maximizer in~\eqref{eq_definition_maxcontsol}.
We can then write that
\begin{equation*}
\Re \beta(w,\LP(\theta))
=
\Re k^2 (\mathfrak{p} v,\theta)
\leq
k^2 \|w\|_{\mathfrak p,\Omega}\|\theta\|_{\mathfrak p,\Omega}
\leq
\mathfrak K \enorm{w}_\Omega
\end{equation*}
for all $w \in H^1_{\GD}(\Omega)$. Using~\eqref{eq_beta_symmetric} and defining
$u \eq \overline{\LP(\theta)}/\enorm{\LP(\theta)}_\Omega = \overline{\LP(\theta)}/\maxcontsol$,
we have
\begin{equation*}
\Re \beta(u,v) = \frac{1}{\maxcontsol} \Re b(\overline{v},\LP(\theta))
\leq
\frac{\mathfrak K}{\Theta} \enorm{v}_\Omega,
\end{equation*}
for all $v \in H^1_{\GD}(\Omega)$ and~\eqref{eq_upper_bound_infsup} follows.
\end{proof}

The assumption that $\beta$ is symmetric holds true for Helmholtz problems.
We could also lift this assumption at the price of also analyzing adjoint
problems. We refrain from doing so here for simplicity. We also recall that
for Helmoltz problem, $\mathfrak K = 1$.

\begin{theorem}[Efficiency of the inf-sup bound]
Assume that $\beta$ is symmetric as per~\eqref{eq_beta_symmetric}.
Then, we have
\begin{equation}
\label{eq_upper_bound}
\gamma
=
\min_{\substack{u \in H^1_{\GD}(\Omega) \\ \enorm{u}_\Omega = 1}}
\max_{\substack{v \in H^1_{\GD}(\Omega) \\ \enorm{v}_\Omega = 1}}
\Re \beta(u,v)
\leq
2\mathfrak K
\iota_h
\gamma_h
\end{equation}
whenever
\begin{equation}
\label{eq_definition_iota}
\iota_h \eq
\frac{1}{ 1 - 2\left (\frac{k\hmax}{\vmin}\right )^2 - 2\maxest}
\left (
1 + \frac{1+2\maxerr}{2\maxcontsol}
\right )
>
0.
\end{equation}
\end{theorem}

\begin{proof}
The estimate in~\eqref{eq_upper_bound_infsup} ensures that
\begin{equation*}
\gamma \leq \frac{\mathfrak K}{\maxcontsol}
\leq
\frac{\mathfrak K}{1+2\maxcontsol +2\maxerr}
\frac{1+2\maxcontsol +2\maxerr}{\maxcontsol}
\leq
\frac{2\mathfrak K}{1+2(\maxcontsol+\maxerr)}
\left (
1+ \frac{1+2\maxerr}{2\maxcontsol}
\right )
\end{equation*}
and it follows from~\eqref{eq_bound_maxsol}  that
\begin{equation*}
\gamma
\leq
\left (
1+ \frac{1+2\maxerr}{2\maxcontsol}
\right )
\frac{2\mathfrak K}{1+2\maxsol}.
\end{equation*}
At that point,~\eqref{eq_upper_bound} follows from
the definitions of~$\iota_h$ in~\eqref{eq_definition_iota}
and~$\gamma_h$ in~\eqref{eq_definition_gamma_h}.
\end{proof}

\begin{remark}[Efficiency for Helmholtz problems]
For Helmholtz problems $M$ is generically bounded and
$\maxcontsol \geq c(\widehat \Omega)k\ell/\vel$,
where $\vel \eq \min_{K \in \CT_h} \vel_K$ is the minimal wavespeed.
Hence, under the assumptions that
\begin{equation*}
\frac{k\ell}{\vel} \gg 1,
\qquad
\frac{k\hmax}{\vmin} \ll 1,
\qquad
\maxerr \ll 1,
\end{equation*}
we have
\begin{equation*}
\iota_h
\leq
1
+
c(\widehat \Omega)
\left (\frac{k\ell}{\vel}\right )^{-1}
+
c(\shapereg)
\LK
\left \{
\left (
\frac{k\hmax}{\vmin}
\right )^2
+
\maxerr
\right \},
\end{equation*}
where $\LK \eq \max_{\ba \in \CV_h} \LK_{\oma}$ is the maximal contrast.
Since we also have $\mathfrak K = 1$, the lower bound provided by the proposed
algorithm is expected to be sharp up to factor $2$ for reasonable discretization settings.
Indeed~\eqref{eq_lower_bound} and~\eqref{eq_upper_bound} can then be simplified into
\begin{equation*}
\gamma_h
\leq
\min_{\substack{u \in H^1_{\GD}(\Omega) \\ \enorm{u}_\Omega = 1}}
\max_{\substack{v \in H^1_{\GD}(\Omega) \\ \enorm{v}_\Omega = 1}}
\Re \beta(u,v)
\leq
2 \left (
1
+
c(\widehat \Omega)
\left (\frac{k\ell}{\vel} \right )^{-1}
+
c(\shapereg)
\LK
\left \{
\left (
\frac{k\hmax}{\vmin}
\right )^2
+
\maxerr
\right \}
\right )
\gamma_h.
\end{equation*}
\end{remark}

\section{Numerical examples}
\label{section_numerical_examples}

\subsection{Generic setting}

In the two examples below, we work in the square $\Omega = (-1,1)^2$
with Dirichlet boundary conditions, i.e. $\GD = \partial \Omega$.
In both cases, we take $\bb = \bc = \bo$ and $\MA = \MI$. We further
set $\mathfrak{m} = \mathfrak{p} = 1$ and $\underline{\mathfrak{A}} = \MI$.

For $N \geq 1$, we consider structured meshes $\CT_h$ with mesh size $h = 1/N$
by first subdividing $\Omega$ into $2N \times 2N$ squares of size $h$,
and then subdividing each square into four triangles by joing its vertices
to its barycenter.

We work with $p=1$, $2$ or $3$. Given $\theta_h \in Q_h$,
the associated flux is obtained by solving the global problem
\begin{equation*}
\BLF_h(\theta_h)
\eq
\arg \min_{\substack{\sig_h \in \RT_{p+1}(\CT_h) \cap \BH(\ddiv,\Omega) \\ \div \sig_h = k^2\theta_h+d_\dagger\LP_h(\theta_h)}}
\|\grad \LP_h(\theta_h)+\sig_h\|_{\Omega}.
\end{equation*}
This construction does not directly fit the framework of Section~\ref{section_upper_bound},
but it is observe to be (globally) efficient in practice.

The singular value problems defining $\Theta_h$ and $\rho_h$ are numerically solved
with a power iteration. Specifically, starting with a randomly initialized $\theta_h^{(0)}$,
we set
\begin{equation*}
\theta_h^{(\ell+1)}
=
\frac{1}{\|(\LP_h^\star \circ \LP_h)(\theta_h^{(\ell)})\|_{\mathfrak m}}
(\LP_h^\star \circ \LP_h)(\theta_h^{(\ell)})
\end{equation*}
where $\LP_h^\star$ is the adjoint of $\LP_h$ for the inner product associated
with the $\enorm{\cdot}_\Omega$ norm, for $0 \leq \ell < 128$, and we set
\begin{equation*}
\widetilde{\Theta}_h \eq \frac{1}{k}\enorm{\LP_h(\theta_h^{(128)})}_\Omega,
\end{equation*}
together with a similar definition for $\widetilde{\rho}_h$. In what follows,
we will employ $\widetilde{\Theta}_h$ and $\widetilde{\rho}_h$ in lieu of
$\Theta_h$ and $\rho_h$ in all relevant formulas, but omit the tildas for
readibility.

We will also consider in the two examples 500 frequencies of the form $k = 2\pi\omega$,
with equally spaced values of $\omega$ ranging from $0.01$ to $5$.

\subsection{A dissipative problem}
\label{section_dissipative_problem}

Here, we first consider the case where $d = 1 + i \tau/k$
with $\tau = 1$. In other words
\begin{equation*}
\beta(u,v) = -k^2(u,v)_\Omega-i\tau k(u,v)_\Omega+(\grad u,\grad v)_\Omega
\end{equation*}
for all $u,v \in H^1_0(\Omega)$. The associated PDE problem is well-posed
for all frequencies, and the inf-sup constant is known to behave
in this case as $\gamma \sim \tau/k$. In fact, $\gamma$ and $\Theta$
are analytically available because the eigenpairs of the Dirichlet
Laplace problem are explicitly known.

On Figure~\ref{figure_dissipative_absolute}, we represent the values
of $\gamma_h$ and $\Theta_h$ computed for different frequencies $k$,
mesh sizes $h$, and polynomial degree $p$. As can be seen on the right-panel,
although we expect $\Theta_h$ to increase linearly with $k$, the curves
fall of for coarse meshes and/or small polynomial degrees as the frequency
increases. This is in fact expected, since coarse discretizations cannot
capture the oscillations leading to the increase in $\Theta$, so that
$\Theta_h$ cannot be trusted for large frequencies and coarse discretizations.
On the left panel of Figure~\ref{figure_dissipative_absolute}, we see
that this behaviour is corrected in $\gamma_h$. Indeed, when $\Theta_h$
underestimates $\Theta$, $\gamma_h$ becomes negative. In other words,
the proposed algorithm never provides inaccurate bounds for $\gamma$, but
rather, it does not provide a bound at all when the discretization is too coarse.

Figure~\ref{figure_dissipative_relative} similarly shows the relative values
$\gamma_h/\gamma$ and $\Theta_h/\Theta$. We see that although $\Theta_h$ can
underestimate $\Theta$, $\gamma_h$ never overestimates $\gamma$, as desired.
We further see that for the finest discretization employed ($p=3$ and $N=32$),
the bound $\gamma_h$ on $\gamma$ is very satisfactory, with an underestimation
factor always less than 2, as predicted by our theoretical analysis.

On Figure~\ref{figure_dissipative_absolute}, we have drawn vertical lines
to indicate, for each discretization setting, the maximal frequency $k$
for which a positive value of $\gamma_h$ is obtained. Similarly, the
vertical lines on Figure~\ref{figure_dissipative_relative} show the
maximal frequency $k$ for which $\gamma_h \geq \gamma/2$. We see that
when refining the mesh or increasing the polynomial degree, these lines
move towards the right of the figure. This means that higher frequencies
are satisfactorily handled as the discretization is refined, which is again
in line with theoretical analysis. We in particular observe the improved
behaviour of increasing the polynomial degree.

Figure~\ref{figure_dissipative_freq_size} finally gives a finer representation
of the improved efficiency of high polynomial degrees. As the different slopes
show, the number of degrees of freedom $N_{\rm dofs}(k)$ (directly linked to $N$)
required to obtain a satisfactory bound for the frequency $k$ increases less
quickly for larger polynomial degree. This is again in line with the intuition
that high-order methods are more performant for high-frequency wave problems.

\begin{figure}
\begin{minipage}{.45\linewidth}
\begin{tikzpicture}
\begin{axis}
[
	width=\linewidth,
	xmin=0,
	xmax=5,
	ymax=1,
	ymin=-1.2,
	xlabel={$\omega$},
	ylabel={$\gamma_h$}
]

\plot[dotted] {0};

\plot [color=black,mark=none] coordinates {(0.535, -1.2) (0.535, 1)};
\plot [color=red  ,mark=none] coordinates {(0.775, -1.2) (0.775, 1)};
\plot [color=blue ,mark=none] coordinates {(1.225, -1.2) (1.225, 1)};
\plot [color=black,mark=none] coordinates {(1.665, -1.2) (1.665, 1)};

\plot[color=black,mark=none] table[x=freq,y=gamma] {figures/data/O0/dissipative/run_0004/dataA.txt} node[pos=0.4,pin=-90:{$N=4$} ] {};
\plot[color=red  ,mark=none] table[x=freq,y=gamma] {figures/data/O0/dissipative/run_0008/dataA.txt}                                  ;
\plot[color=blue ,mark=none] table[x=freq,y=gamma] {figures/data/O0/dissipative/run_0016/dataA.txt}                                  ;
\plot[color=black,mark=none] table[x=freq,y=gamma] {figures/data/O0/dissipative/run_0032/dataA.txt} node[pos=0.8,pin= 90:{$N=32$}] {};

\end{axis}
\end{tikzpicture}
\end{minipage}
\begin{minipage}{.45\linewidth}
\begin{tikzpicture}
\begin{axis}
[
	width=\linewidth,
	xmin=0,
	xmax=5,
	ymin=0,
	ymax=32,
	xlabel={$\omega$},
	ylabel={$\Theta_h$}
]

\plot[color=black,mark=none] table[x=freq,y=theta] {figures/data/O0/dissipative/run_0004/dataA.txt} node[pos=0.8,pin={[pin distance=0.5cm]- 90:{$N=4$}} ] {};
\plot[color=red  ,mark=none] table[x=freq,y=theta] {figures/data/O0/dissipative/run_0008/dataA.txt}                                                       ;
\plot[color=blue ,mark=none] table[x=freq,y=theta] {figures/data/O0/dissipative/run_0016/dataA.txt}                                                       ;
\plot[color=black,mark=none] table[x=freq,y=theta] {figures/data/O0/dissipative/run_0032/dataA.txt} node[pos=0.8,pin={[pin distance=2cm] 180:{$N=32$}}] {};

\end{axis}
\end{tikzpicture}
\end{minipage}

$p=1$

\begin{minipage}{.45\linewidth}
\begin{tikzpicture}
\begin{axis}
[
	width=\linewidth,
	xmin=0,
	xmax=5,
	ymax=1,
	ymin=-1.2,
	xlabel={$\omega$},
	ylabel={$\gamma_h$}
]

\plot[dotted] {0};

\plot [color=black,mark=none] coordinates {(1.025, -1.2) (1.025, 1)};
\plot [color=red  ,mark=none] coordinates {(1.745, -1.2) (1.745, 1)};
\plot [color=blue ,mark=none] coordinates {(2.975, -1.2) (2.975, 1)};
\plot [color=black,mark=none] coordinates {(4.745, -1.2) (4.745, 1)};

\plot[color=black,mark=none] table[x=freq,y=gamma] {figures/data/O1/dissipative/run_0004/dataA.txt} node[pos=0.45,pin=-90:{$N=4$} ] {};
\plot[color=red  ,mark=none] table[x=freq,y=gamma] {figures/data/O1/dissipative/run_0008/dataA.txt}                                   ;
\plot[color=blue ,mark=none] table[x=freq,y=gamma] {figures/data/O1/dissipative/run_0016/dataA.txt}                                   ;
\plot[color=black,mark=none] table[x=freq,y=gamma] {figures/data/O1/dissipative/run_0032/dataA.txt} node[pos=0.80,pin= 90:{$N=32$}] {};

\end{axis}
\end{tikzpicture}
\end{minipage}
\begin{minipage}{.45\linewidth}
\begin{tikzpicture}
\begin{axis}
[
	width=\linewidth,
	xmin=0,
	xmax=5,
	ymin=0,
	ymax=32,
	xlabel={$\omega$},
	ylabel={$\Theta_h$}
]

\plot[color=black,mark=none] table[x=freq,y=theta] {figures/data/O1/dissipative/run_0004/dataA.txt} node[pos=0.8,pin={[pin distance=2cm]- 90:{$N=4$}} ] {};
\plot[color=red  ,mark=none] table[x=freq,y=theta] {figures/data/O1/dissipative/run_0008/dataA.txt}                                                       ;
\plot[color=blue ,mark=none] table[x=freq,y=theta] {figures/data/O1/dissipative/run_0016/dataA.txt}                                                       ;
\plot[color=black,mark=none] table[x=freq,y=theta] {figures/data/O1/dissipative/run_0032/dataA.txt} node[pos=0.8,pin={[pin distance=2cm] 180:{$N=32$}}] {};

\end{axis}
\end{tikzpicture}
\end{minipage}

$p=2$

\begin{minipage}{.45\linewidth}
\begin{tikzpicture}
\begin{axis}
[
	width=\linewidth,
	xmin=0,
	xmax=5,
	ymax=1,
	ymin=-1.2,
	xlabel={$\omega$},
	ylabel={$\gamma_h$}
]

\plot[dotted] {0};

\plot [color=black,mark=none] coordinates {(1.325, -1.2) (1.325, 1)};
\plot [color=red  ,mark=none] coordinates {(2.495, -1.2) (2.495, 1)};
\plot [color=blue ,mark=none] coordinates {(4.655, -1.2) (4.655, 1)};

\plot[color=black,mark=none] table[x=freq,y=gamma] {figures/data/O2/dissipative/run_0004/dataA.txt} node[pos=0.4,pin=-90:{$N=4$} ] {};
\plot[color=red  ,mark=none] table[x=freq,y=gamma] {figures/data/O2/dissipative/run_0008/dataA.txt}                                  ;
\plot[color=blue ,mark=none] table[x=freq,y=gamma] {figures/data/O2/dissipative/run_0016/dataA.txt}                                  ;
\plot[color=black,mark=none] table[x=freq,y=gamma] {figures/data/O2/dissipative/run_0032/dataA.txt} node[pos=0.8,pin= 90:{$N=32$}] {};

\end{axis}
\end{tikzpicture}
\end{minipage}
\begin{minipage}{.45\linewidth}
\begin{tikzpicture}
\begin{axis}
[
	width=\linewidth,
	xmin=0,
	xmax=5,
	ymin=0,
	ymax=32,
	xlabel={$\omega$},
	ylabel={$\Theta_h$}
]

\plot[color=black,mark=none] table[x=freq,y=theta] {figures/data/O2/dissipative/run_0004/dataA.txt} node[pos=0.8,pin={[pin distance=2cm]- 90:{$N=4$}} ] {};
\plot[color=red  ,mark=none] table[x=freq,y=theta] {figures/data/O2/dissipative/run_0008/dataA.txt}                                                       ;
\plot[color=blue ,mark=none] table[x=freq,y=theta] {figures/data/O2/dissipative/run_0016/dataA.txt}                                                       ;
\plot[color=black,mark=none] table[x=freq,y=theta] {figures/data/O2/dissipative/run_0032/dataA.txt} node[pos=0.8,pin={[pin distance=2cm] 180:{$N=32$}}] {};

\end{axis}
\end{tikzpicture}
\end{minipage}

$p=3$

\caption{Absolute behaviours of $\gamma_h$ and $\Theta_h$ in the dissipative example.}
\label{figure_dissipative_absolute}
\end{figure}

\begin{figure}
\begin{minipage}{.45\linewidth}
\begin{tikzpicture}
\begin{axis}
[
	width=\linewidth,
	xmin=0,
	xmax=5,
	ymax=1,
	ymin=-1,
	xlabel={$\omega$},
	ylabel={$\gamma_h/\gamma$}
]

\plot[dotted] {0.5};

\plot [color=black,mark=none] coordinates {(0.345, -1.2) (0.345, 1)};
\plot [color=red  ,mark=none] coordinates {(0.515, -1.2) (0.515, 1)};
\plot [color=blue ,mark=none] coordinates {(0.765, -1.2) (0.765, 1)};
\plot [color=black,mark=none] coordinates {(1.205, -1.2) (1.205, 1)};

\plot[color=black,mark=none] table[x=freq,y expr=\thisrow{gamma}/\thisrow{gammaA}] {figures/data/O0/dissipative/run_0004/dataA.txt} node[pos=0.007,pin={[pin distance={2cm}]0:{$N=4$}}] {};
\plot[color=red  ,mark=none] table[x=freq,y expr=\thisrow{gamma}/\thisrow{gammaA}] {figures/data/O0/dissipative/run_0008/dataA.txt}                                                      ;
\plot[color=blue ,mark=none] table[x=freq,y expr=\thisrow{gamma}/\thisrow{gammaA}] {figures/data/O0/dissipative/run_0016/dataA.txt}                                                      ;
\plot[color=black,mark=none] table[x=freq,y expr=\thisrow{gamma}/\thisrow{gammaA}] {figures/data/O0/dissipative/run_0032/dataA.txt} node[pos=0.08,pin=45:{$N=32$}] {}                    ;

\end{axis}
\end{tikzpicture}
\end{minipage}
\begin{minipage}{.45\linewidth}
\begin{tikzpicture}
\begin{axis}
[
	width=\linewidth,
	xmin=0,
	xmax=5,
	ymin=0.1,
	ymax=10.,
	ymode=log,
	xlabel={$\omega$},
	ylabel={$\Theta_h/\Theta$}
]

\plot[color=black,mark=none] table[x=freq,y expr=\thisrow{theta}/\thisrow{thetaA}] {figures/data/O0/dissipative/run_0004/dataA.txt} node[pos=0.8,pin=-90:{$N=4$} ] {};;
\plot[color=red  ,mark=none] table[x=freq,y expr=\thisrow{theta}/\thisrow{thetaA}] {figures/data/O0/dissipative/run_0008/dataA.txt}                                  ;;
\plot[color=blue ,mark=none] table[x=freq,y expr=\thisrow{theta}/\thisrow{thetaA}] {figures/data/O0/dissipative/run_0016/dataA.txt}                                  ;;
\plot[color=black,mark=none] table[x=freq,y expr=\thisrow{theta}/\thisrow{thetaA}] {figures/data/O0/dissipative/run_0032/dataA.txt} node[pos=0.8,pin= 90:{$N=32$}] {};;

\end{axis}
\end{tikzpicture}
\end{minipage}

$p=1$

\begin{minipage}{.45\linewidth}
\begin{tikzpicture}
\begin{axis}
[
	width=\linewidth,
	xmin=0,
	xmax=5,
	ymax=1,
	ymin=-1,
	xlabel={$\omega$},
	ylabel={$\gamma_h/\gamma$}
]

\plot[dotted] {0.5};

\plot [color=black,mark=none] coordinates {(0.755, -1.2) (0.755, 1)};
\plot [color=red  ,mark=none] coordinates {(1.345, -1.2) (1.345, 1)};
\plot [color=blue ,mark=none] coordinates {(2.235, -1.2) (2.235, 1)};
\plot [color=black,mark=none] coordinates {(3.725, -1.2) (3.725, 1)};

\plot[color=black,mark=none] table[x=freq,y expr=\thisrow{gamma}/\thisrow{gammaA}] {figures/data/O1/dissipative/run_0004/dataA.txt} node[pos=0.01,pin={[pin distance=0.5cm]-100:{\begin{turn}{90}$N=4$\end{turn}}}] {};
\plot[color=red  ,mark=none] table[x=freq,y expr=\thisrow{gamma}/\thisrow{gammaA}] {figures/data/O1/dissipative/run_0008/dataA.txt}                                    ;
\plot[color=blue ,mark=none] table[x=freq,y expr=\thisrow{gamma}/\thisrow{gammaA}] {figures/data/O1/dissipative/run_0016/dataA.txt}                                    ;
\plot[color=black,mark=none] table[x=freq,y expr=\thisrow{gamma}/\thisrow{gammaA}] {figures/data/O1/dissipative/run_0032/dataA.txt} node[pos=0.80,pin=  90:{$N=32$}] {};

\end{axis}
\end{tikzpicture}
\end{minipage}
\begin{minipage}{.45\linewidth}
\begin{tikzpicture}
\begin{axis}
[
	width=\linewidth,
	xmin=0,
	xmax=5,
	ymin=0.1,
	ymax=10,
	ymode=log,
	xlabel={$\omega$},
	ylabel={$\Theta_h/\Theta$}
]

\plot[color=black,mark=none] table[x=freq,y expr=\thisrow{theta}/\thisrow{thetaA}] {figures/data/O1/dissipative/run_0004/dataA.txt} node[pos=0.8,pin=-90:{$N=4$} ] {};;
\plot[color=red  ,mark=none] table[x=freq,y expr=\thisrow{theta}/\thisrow{thetaA}] {figures/data/O1/dissipative/run_0008/dataA.txt}                                  ;;
\plot[color=blue ,mark=none] table[x=freq,y expr=\thisrow{theta}/\thisrow{thetaA}] {figures/data/O1/dissipative/run_0016/dataA.txt}                                  ;;
\plot[color=black,mark=none] table[x=freq,y expr=\thisrow{theta}/\thisrow{thetaA}] {figures/data/O1/dissipative/run_0032/dataA.txt} node[pos=0.8,pin= 90:{$N=32$}] {};;

\end{axis}
\end{tikzpicture}
\end{minipage}

$p=2$

\begin{minipage}{.45\linewidth}
\begin{tikzpicture}
\begin{axis}
[
	width=\linewidth,
	xmin=0,
	xmax=5,
	ymax=1,
	ymin=-1,
	xlabel={$\omega$},
	ylabel={$\gamma_h/\gamma$}
]

\plot[dotted] {0.5};

\plot [color=black,mark=none] coordinates {(0.905, -1.2) (0.905, 1)};
\plot [color=red  ,mark=none] coordinates {(1.855, -1.2) (1.855, 1)};
\plot [color=blue ,mark=none] coordinates {(3.535, -1.2) (3.535, 1)};

\plot[color=black,mark=none] table[x=freq,y expr=\thisrow{gamma}/\thisrow{gammaA}] {figures/data/O2/dissipative/run_0004/dataA.txt} node[pos=0.01,pin={[pin distance=2cm]-90:{$N=4$}} ] {};
\plot[color=red  ,mark=none] table[x=freq,y expr=\thisrow{gamma}/\thisrow{gammaA}] {figures/data/O2/dissipative/run_0008/dataA.txt}                                                        ;
\plot[color=blue ,mark=none] table[x=freq,y expr=\thisrow{gamma}/\thisrow{gammaA}] {figures/data/O2/dissipative/run_0016/dataA.txt}                                                        ;
\plot[color=black,mark=none] table[x=freq,y expr=\thisrow{gamma}/\thisrow{gammaA}] {figures/data/O2/dissipative/run_0032/dataA.txt} node[pos=0.850,pin={[pin distance=2cm]-90:{$N=32$}}] {};

\end{axis}
\end{tikzpicture}
\end{minipage}
\begin{minipage}{.45\linewidth}
\begin{tikzpicture}
\begin{axis}
[
	width=\linewidth,
	xmin=0,
	xmax=5,
	ymode=log,
	ymin=0.1,
	ymax=10.,
	xlabel={$\omega$},
	ylabel={$\Theta_h/\Theta$}
]

\plot[color=black,mark=none] table[x=freq,y expr=\thisrow{theta}/\thisrow{thetaA}] {figures/data/O2/dissipative/run_0004/dataA.txt} node[pos=0.8,pin=-90:{$N=4$} ] {};
\plot[color=red  ,mark=none] table[x=freq,y expr=\thisrow{theta}/\thisrow{thetaA}] {figures/data/O2/dissipative/run_0008/dataA.txt}                                  ;
\plot[color=blue ,mark=none] table[x=freq,y expr=\thisrow{theta}/\thisrow{thetaA}] {figures/data/O2/dissipative/run_0016/dataA.txt}                                  ;
\plot[color=black,mark=none] table[x=freq,y expr=\thisrow{theta}/\thisrow{thetaA}] {figures/data/O2/dissipative/run_0032/dataA.txt} node[pos=0.8,pin= 90:{$N=32$}] {};

\end{axis}
\end{tikzpicture}
\end{minipage}

$p=3$

\caption{Relative behaviours of $\gamma_h$ and $\Theta_h$ in the dissipative example.}
\label{figure_dissipative_relative}
\end{figure}

\begin{figure}
\begin{minipage}{.45\linewidth}
\begin{tikzpicture}
\begin{axis}
[
	width=\linewidth,
	xmode=log,
	ymode=log,
	ymin=0.3,
	ymax=5.0,
	xtick={4,8,16,32},
	xticklabels={4,8,16,32},
	ytick={0.5,1,2,4},
	yticklabels={0.5,1,2,4},
	ylabel={$\omega$},
	xlabel={$N$}
]

\plot[black,mark=o     ] table[x=mesh,y=freq] {figures/data/O0/dissipative/scale.txt} node[pos=0.9,pin=-90:{$p=1$}] {};
\plot[red  ,mark=x     ] table[x=mesh,y=freq] {figures/data/O1/dissipative/scale.txt} node[pos=0.6,pin=  0:{$p=2$}] {};
\plot[blue ,mark=square] table[x=mesh,y=freq] {figures/data/O2/dissipative/scale.txt} node[pos=0.9,pin=180:{$p=3$}] {};
\end{axis}
\end{tikzpicture}
\end{minipage}
\begin{minipage}{.45\linewidth}
\begin{tikzpicture}
\begin{axis}
[
	width=\linewidth,
	xmode=log,
	ymode=log,
	ymin=0.3,
	ymax=5.0,
	xtick={4,8,16,32},
	xticklabels={4,8,16,32},
	ytick={0.5,1,2,4},
	yticklabels={0.5,1,2,4},
	ylabel={$\omega$},
	xlabel={$N$}
]

\plot[black,mark=o     ] table[x=mesh,y=freqH] {figures/data/O0/dissipative/scale.txt} node[pos=0.9,pin=-90:{$p=1$}] {};
\plot[red  ,mark=x     ] table[x=mesh,y=freqH] {figures/data/O1/dissipative/scale.txt} node[pos=0.6,pin=  0:{$p=2$}] {};
\plot[blue ,mark=square] table[x=mesh,y=freqH] {figures/data/O2/dissipative/scale.txt} node[pos=0.9,pin=180:{$p=3$}] {};

\end{axis}
\end{tikzpicture}
\end{minipage}
\caption{Maximal frequency in the dissipative example for which $\gamma_h > 0$ on the left panel and $\gamma_h \geq \gamma/2$ on the right panel.
(There are only three data points for $p=3$ because the thresholds are never reached for $N=32$.)}
\label{figure_dissipative_freq_size}
\end{figure}

\subsection{A cavity problem}

We now consider the case where $d=1$, i.e. the setting is the same as above with $\tau = 0$.
In this case, there is no absorption, and the problem at hand is not well-posed for all
frequencies. More precisely, well-posedness fails whenever $k = (\pi/2)\sqrt{n^2+m^2}$
some positive integers $n,m$. There are 131 such resonant frequencies (counted without
multiplicity) in the range $[0,5] \cdot 2\pi$ and $5$ of them exactly belong to the
sampled values, namely $\{1.25,2.5,3.75,4.25,5\} \cdot 2\pi$.

Figure~\ref{figure_cavity_gamma} graphs the computed values of $\gamma_h$ and the ratio
$\gamma/\gamma_h$. These graphs are harder to decipher than the one from
Section~\ref{section_dissipative_problem} due to the large number of resonant frequencies
within the considered range. We can nevertheless conclude that, as proved, the algorithm always
provides a guaranteed bound $\gamma \geq \gamma_h$. We further see that this lower bound is not
overly pessimistic if the discretization is fine.

\begin{figure}
\begin{minipage}{.45\linewidth}
\begin{tikzpicture}
\begin{axis}
[
	width=\linewidth,
	xmin=0,
	xmax=5,
	ymax=1,
	ymin=-1.,
	xlabel={$\omega$},
	ylabel={$\gamma_h$}
]

\plot[dotted] {0};

\plot[color=black,mark=none] table[x=freq,y=gamma] {figures/data/O0/cavity/run_0004/dataA.txt} node[pos=0.3,pin=-90:{$N=4$} ] {};
\plot[color=red  ,mark=none] table[x=freq,y=gamma] {figures/data/O0/cavity/run_0008/dataA.txt}                                  ;
\plot[color=blue ,mark=none] table[x=freq,y=gamma] {figures/data/O0/cavity/run_0016/dataA.txt}                                  ;
\plot[color=black,mark=none] table[x=freq,y=gamma] {figures/data/O0/cavity/run_0032/dataA.txt} node[pos=0.8,pin= 90:{$N=32$}] {};

\end{axis}
\end{tikzpicture}
\end{minipage}
\begin{minipage}{.45\linewidth}
\begin{tikzpicture}
\begin{axis}
[
	width=\linewidth,
	xmin=0,
	xmax=5,
	ymax=1,
	ymin=-1,
	xlabel={$\omega$},
	ylabel={$\gamma_h/\gamma$}
]

\plot[dotted] {0.};

\plot[color=black,mark=none] table[x=freq,y=ratio] {figures/data/O0/cavity/run_0004/ratio.txt} node[pos=0.01,pin= 0:{$N=4$} ] {};
\plot[color=red  ,mark=none] table[x=freq,y=ratio] {figures/data/O0/cavity/run_0008/ratio.txt}                                   ;
\plot[color=blue ,mark=none] table[x=freq,y=ratio] {figures/data/O0/cavity/run_0016/ratio.txt}                                   ;
\plot[color=black,mark=none] table[x=freq,y=ratio] {figures/data/O0/cavity/run_0032/ratio.txt} node[pos=0.60,pin=   0:{$N=32$}] {};

\end{axis}
\end{tikzpicture}
\end{minipage}

$p=1$

\begin{minipage}{.45\linewidth}
\begin{tikzpicture}
\begin{axis}
[
	width=\linewidth,
	xmin=0,
	xmax=5,
	ymax=1,
	ymin=-1.,
	xlabel={$\omega$},
	ylabel={$\gamma_h$}
]

\plot[dotted] {0};

\plot[color=black,mark=none] table[x=freq,y=gamma] {figures/data/O1/cavity/run_0004/dataA.txt} node[pos=0.3,pin=-90:{$N=4$} ] {};
\plot[color=red  ,mark=none] table[x=freq,y=gamma] {figures/data/O1/cavity/run_0008/dataA.txt}                                  ;
\plot[color=blue ,mark=none] table[x=freq,y=gamma] {figures/data/O1/cavity/run_0016/dataA.txt}                                  ;
\plot[color=black,mark=none] table[x=freq,y=gamma] {figures/data/O1/cavity/run_0032/dataA.txt} node[pos=0.8,pin= 90:{$N=32$}] {};

\end{axis}
\end{tikzpicture}
\end{minipage}
\begin{minipage}{.45\linewidth}
\begin{tikzpicture}
\begin{axis}
[
	width=\linewidth,
	xmin=0,
	xmax=5,
	ymax=1,
	ymin=-1,
	xlabel={$\omega$},
	ylabel={$\gamma_h/\gamma$}
]

\plot[dotted] {0.};

\plot[color=black,mark=none] table[x=freq,y=ratio] {figures/data/O1/cavity/run_0004/ratio.txt} node[pos=0.03,pin=-90:{\begin{turn}{90}$N=4$\end{turn}} ] {};
\plot[color=red  ,mark=none] table[x=freq,y=ratio] {figures/data/O1/cavity/run_0008/ratio.txt}                                   ;
\plot[color=blue ,mark=none] table[x=freq,y=ratio] {figures/data/O1/cavity/run_0016/ratio.txt}                                   ;
\plot[color=black,mark=none] table[x=freq,y=ratio] {figures/data/O1/cavity/run_0032/ratio.txt} node[pos=0.7,pin=45:{$N=32$}] {};

\end{axis}
\end{tikzpicture}
\end{minipage}

$p=2$

\begin{minipage}{.45\linewidth}
\begin{tikzpicture}
\begin{axis}
[
	width=\linewidth,
	xmin=0,
	xmax=5,
	ymax=1,
	ymin=-1,
	xlabel={$\omega$},
	ylabel={$\gamma_h$}
]

\plot[dotted] {0};

\plot[color=black,mark=none] table[x=freq,y=gamma] {figures/data/O2/cavity/run_0004/dataA.txt} node[pos=0.3,pin=-90:{$N=4$} ] {};
\plot[color=red  ,mark=none] table[x=freq,y=gamma] {figures/data/O2/cavity/run_0008/dataA.txt}                                  ;
\plot[color=blue ,mark=none] table[x=freq,y=gamma] {figures/data/O2/cavity/run_0016/dataA.txt}                                  ;
\plot[color=black,mark=none] table[x=freq,y=gamma] {figures/data/O2/cavity/run_0032/dataA.txt} node[pos=0.8,pin= 90:{$N=32$}] {};

\end{axis}
\end{tikzpicture}
\end{minipage}
\begin{minipage}{.45\linewidth}
\begin{tikzpicture}
\begin{axis}
[
	width=\linewidth,
	xmin=0,
	xmax=5,
	ymax=1,
	ymin=-1,
	xlabel={$\omega$},
	ylabel={$\gamma_h/\gamma$}
]

\plot[dotted] {0.};

\plot[color=black,mark=none] table[x=freq,y=ratio] {figures/data/O2/cavity/run_0004/ratio.txt} node[pos=0.05,pin=-90:{\begin{turn}{90}$N=4$\end{turn}} ] {};
\plot[color=red  ,mark=none] table[x=freq,y=ratio] {figures/data/O2/cavity/run_0008/ratio.txt}                                   ;
\plot[color=blue ,mark=none] table[x=freq,y=ratio] {figures/data/O2/cavity/run_0016/ratio.txt}                                   ;
\plot[color=black,mark=none] table[x=freq,y=ratio] {figures/data/O2/cavity/run_0032/ratio.txt} node[pos=0.75,pin={[pin distance=1.1cm]90:{$N=32${\color{white}xxx}}}] {};

\end{axis}
\end{tikzpicture}
\end{minipage}

$p=3$

\caption{Behaviour of $\gamma_h$ in the cavity example.}
\label{figure_cavity_gamma}
\end{figure}

On Figure~\ref{figure_table}, we consider the finest discretization ($p=3$ and $N=32$),
and list the frequencies $\omega_h$ for which $\gamma_h \leq 0$. As can be seen there, all these
frequencies have a true resonant frequency $\omega$ such that $|\omega-\omega_h| \leq 5 \cdot 10^{-3}$.
In other words, we only obtain ``false negative'' for frequencies close to resonant values,
which is the desired behaviour.

\begin{figure}
\begin{center}

{
\footnotesize

\begin{tabular}{|ccc||ccc||ccc||ccc|}
\hline
$\omega_h$ & $\omega$ & $|\omega-\omega_h|$ &
$\omega_h$ & $\omega$ & $|\omega-\omega_h|$ &
$\omega_h$ & $\omega$ & $|\omega-\omega_h|$ &
$\omega_h$ & $\omega$ & $|\omega-\omega_h|$
\\
\hline
1.25 & 1.2500 & 0.00e+00 & 3.76 & 3.7583 & 1.68e-03 & 4.28 & 4.2793 & 6.89e-04 & 4.74 & 4.7434 & 3.42e-03
\\
1.82 & 1.8200 & 2.75e-05 & 3.88 & 3.8810 & 1.04e-03 & 4.30 & 4.3012 & 1.16e-03 & 4.76 & 4.7566 & 3.43e-03
\\
2.50 & 2.5000 & 0.00e+00 & 3.89 & 3.8891 & 9.13e-04 & 4.43 & 4.4300 & 1.13e-05 & 4.78 & 4.7762 & 3.76e-03
\\
2.61 & 2.6101 & 7.66e-05 & 4.01 & 4.0078 & 2.20e-03 & 4.45 & 4.4511 & 1.12e-03 & 4.80 & 4.8023 & 2.34e-03
\\
2.85 & 2.8504 & 4.39e-04 & 4.03 & 4.0311 & 1.13e-03 & 4.47 & 4.4721 & 2.14e-03 & 4.81 & 4.8088 & 1.15e-03
\\
3.01 & 3.0104 & 3.99e-04 & 4.04 & 4.0389 & 1.13e-03 & 4.51 & 4.5069 & 3.06e-03 & 4.83 & 4.8283 & 1.70e-03
\\
3.25 & 3.2500 & 5.65e-16 & 4.07 & 4.0697 & 2.95e-04 & 4.53 & 4.5277 & 2.31e-03 & 4.85 & 4.8541 & 4.12e-03
\\
3.26 & 3.2596 & 3.99e-04 & 4.10 & 4.1003 & 3.05e-04 & 4.56 & 4.5621 & 2.07e-03 & 4.91 & 4.9117 & 1.72e-03
\\
3.40 & 3.4004 & 3.68e-04 & 4.14 & 4.1382 & 1.76e-03 & 4.59 & 4.5894 & 6.10e-04 & 4.92 & 4.9244 & 4.43e-03
\\
3.51 & 3.5089 & 1.08e-03 & 4.16 & 4.1608 & 8.29e-04 & 4.61 & 4.6098 & 2.28e-04 & 4.93 & 4.9308 & 7.71e-04
\\
3.58 & 3.5795 & 5.45e-04 & 4.19 & 4.1908 & 7.64e-04 & 4.65 & 4.6503 & 2.69e-04 & 4.95 & 4.9497 & 2.53e-04
\\
3.64 & 3.6401 & 5.49e-05 & 4.25 & 4.2500 & 0.00e+00 & 4.67 & 4.6704 & 3.85e-04 & 4.96 & 4.9624 & 2.36e-03
\\
3.69 & 3.6912 & 1.21e-03 & 4.26 & 4.2573 & 2.65e-03 & 4.70 & 4.6971 & 2.93e-03 & 4.98 & 4.9812 & 1.21e-03
\\
3.75 & 3.7500 & 0.00e+00 & 4.27 & 4.2720 & 2.00e-03 & 4.72 & 4.7170 & 3.01e-03 & 5.00 & 5.0000 & 0.00e+00
\\
\hline
\end{tabular}

}
\end{center}

\caption{Frequencies $\omega_h$ for which $\gamma_h \leq 0$ in the case $p=3$ and $N=32$, closest
true resonant frequency $\omega$, and distance. The largest distance in the table is 4.43e-03.}
\label{figure_table}
\end{figure}

\bibliographystyle{amsplain}
\bibliography{bibliography}

\end{document}